\documentclass[12pt]{article}

\usepackage[a4paper, left=2cm, right=2cm, top=2cm, bottom=2cm]{geometry}

\usepackage{amssymb,amsfonts,amsmath,amsthm,amscd,mathabx,array,hhline}
\usepackage[all]{xy}
\usepackage{amsmath}
\usepackage{authblk}
\usepackage{graphicx}
\usepackage{cite}
\usepackage[table,xcdraw]{xcolor}
\usepackage{multirow}
\usepackage{hyperref}
\usepackage{verbatim}
\hypersetup{hidelinks}

\newtheorem{lemma}{Lemma}[section]
\newtheorem{theorem}[lemma]{Theorem}
\newtheorem{prop}[lemma]{Proposition}

\newtheorem{cor}[lemma]{Corollary}
\newtheorem{conj}[lemma]{Conjecture}
\theoremstyle{definition}

\newtheorem{remark}[lemma]{Remark}

\newenvironment{proof_of}[1]{\medskip\noindent{\it Proof #1}}
{\hfill$\Box$ \bigskip}

\makeatletter
\newcounter{bwt}

\makeatother

\DeclareMathOperator{\Id}{\mathrm{Id}}

\DeclareMathOperator{\Sym}{\mathcal{S}}
\DeclareMathOperator{\A}{\mathcal{A}}

\hypersetup{
  pdftitle={On multilinear polynomial identities for 2 by 2 matrices in characteristic 2},
  pdfauthor={Iritan Ferreira dos Santos},
  pdfkeywords={polynomial identities, matrix algebras,
  multilinear identities, relatively free algebras}
}

\title{\vspace{-24pt}
\bfseries\Large
On multilinear polynomial identities \\
for 
$2 \times 2$ matrices in characteristic $2$}

\author[1]{\bfseries \normalsize Iritan Ferreira dos Santos}

\date{\small }

\affil[1]{Universidade Estadual de Campinas (UNICAMP), 651 Sergio Buarque de Holanda, 13083-859 Campinas, SP, Brazil\\
\texttt{i195167@dac.unicamp.br}}

\begin{document}

\maketitle

\begin{abstract}
\noindent
Let~$K$ be an infinite field of characteristic~$2$.
It is well known that the Lie algebra $\mathfrak{gl}_2(K)=M_2(K)^{(-)}$ of all
$2\times 2$ matrices 
does not admit a finite basis of polynomial identities. Similarly, the finite basis problem for the associative algebra
$M_2(K)$ remains open.
In this note, we obtain two multilinear identities of
degree~$5$ for $M_2(K)$ and
prove that, together with the standard polynomial of degree~$4$, they
generate all multilinear identities of degree up to~$7$. Moreover, we prove that none of the three generators can be omitted.


\end{abstract}


\section{Introduction}

Let $M_n(K)$ be the algebra of 
$n \times n$ matrices over a field 
$K$. The description of the T-ideal of polynomial identities of matrix algebras is one of the central problems in the theory of algebras with polynomial identities (PI-theory). Even for $M_2(K)$, the answer depends essentially on the characteristic of the field.
For
fields of characteristic different from $2$, finite generating sets for the corresponding $T$-ideal are known (see, for example,~\cite{Drensky_1981_2x2, Koshlukov_2001_2x2}).
However, the corresponding problem over an infinite field of characteristic $2$ remains open.

The exceptional behavior in characteristic $2$ is already evident when $M_2(K)$ is viewed as a Lie algebra under the commutator bracket. Vaughan-Lee \cite{Vaughan-Lee1970} proved that the Lie algebra $\mathfrak{gl}_2(K)=M_2(K)^{(-)}$ of $2\times 2$ matrices over an infinite field of characteristic~$2$
does not admit a finite basis of polynomial identities. An explicit infinite basis was obtained by Drensky in his doctoral thesis~\cite{Drensky1979} and, independently, by Lopatin \cite{lopatin2016identitiesliealgebragl2}. Motivated by the Lie case, Drensky formulated the following conjecture:
\begin{conj}[\!\!\cite{Drensky2022}, Conjecture~$1$]
The polynomial identities of $M_2(K)$ over an infinite field of characteristic~$2$ do not follow from a finite number.    
\end{conj}

Although Drensky's conjecture concerns the $T$-ideal of all polynomial identities for $M_2(K)$, this note focuses on its low-degree multilinear components. In positive characteristic, these are logically distinct problems, since polynomial identities cannot, in general, be recovered from their multilinearizations; therefore, even a finite generating set for all multilinear identities would not necessarily provide a generating set for the full $T$-ideal.

Over an infinite field $K$ of characteristic $2$, the low-degree multilinear components of the  polynomial identities for $M_2(K)$ were investigated by Drensky and Tsiganchev \cite{Drensky&Tsiganchev1998} and, subsequently, by Asparouhov, Drensky, Koev, and Tsiganchev \cite{Asparouhov&Drensky&Koev&Tsiganchev2000}. In particular, a multilinear identity $f$ of degree $5$ for $M_2(K)$ was obtained in \cite{Asparouhov&Drensky&Koev&Tsiganchev2000}. However, to the best of our knowledge, an explicit generating set for all multilinear identities up to a certain bounded degree does not appear in the literature.
The aim of this note is to present a complete description of the multilinear identities of degrees up to $7$ for $M_2(K)$.
Our main result is as follows.

\begin{theorem}~\label{theo-01}
Let $K$ be an infinite field of characteristic~$2$. Every multilinear
polynomial identity of $M_2(K)$ of degree up to $7$ is a consequence
of the following three identities:
\begin{align}
 \mathrm{st}_4\left(x_1, x_2,x_3,x_4\right)&=\sum_{\sigma \in \Sym_4} \operatorname{sign}(\sigma) x_{\sigma(1)}x_{\sigma(2)}x_{\sigma(3)}x_{\sigma(4)},  \\  
 \mathfrak{S}(x_1,x_2,x_3,x_4,x_5)&=[[x_1,x_2][x_3,x_5],x_4]+[[x_1,x_3][x_4,x_5],x_2]+[[x_1,x_4][x_2,x_5],x_3],\label{ident-S} \\
\mathfrak{T}(x_1,x_2,x_3,x_4,x_5)
&=
[[x_1,x_2],x_3][x_4,x_5]+
[[x_1,x_2],x_4][x_5,x_3]+
[[x_1,x_2],x_5][x_3,x_4],\label{ident-T}
\end{align}
where $\Sym_4$ is the symmetric group acting on the set $\{1, 2,3,4\}$ and
$[x_1,x_2]=x_1x_2-x_2x_1$ is the commutator of $x_1,x_2$. Moreover, none of these identities can be omitted.
\end{theorem}

We also compare the polynomials $\mathfrak S$ and $\mathfrak T$ with the identity $f$ obtained in \cite{Asparouhov&Drensky&Koev&Tsiganchev2000}. We prove that $f$ is a consequence of $\mathfrak S$ and, furthermore, that $\mathfrak S$ is equivalent to $f$ modulo the polynomial $\mathfrak T$. Consequently, $\{\operatorname{st}_4,f,\mathfrak T\}$ forms an alternative minimal generating set for the multilinear identities of degree up to $7$.

\section{Preliminaries}\label{section-preliminaries}

Throughout this note, unless otherwise stated, $K$ denotes an
infinite field of characteristic $2$.
All vector spaces and algebras are over $K$, and all algebras are
assumed to be associative.

We write $K\langle x_1, \ldots, x_n \rangle$ for the free algebra with free generators $x_1,\ldots, x_n$. In case the set of free generators is infinite and enumerable, and denoted by
$X = \{x_1, x_2, \ldots \}$, the corresponding free algebra is denoted by $K\langle X \rangle$.
Recall that a polynomial 
$f(x_1, \ldots, x_n)$ in $K\langle X \rangle$ is a \textit{polynomial identity}
for an
algebra $\A$ if $f(a_1, \ldots, a_n) = 0$ for all $a_1, \ldots, a_n\in \A$.
If $\A$ satisfies 
a nontrivial polynomial
identity, we call $\A$ an
\textit{algebra with polynomial identity (PI-algebra)}. The set 
$\Id_{K}(\A) = \Id(\A)$ of
all polynomial identities for $\A$ is a 
T-ideal, i.e., $\Id(\A)$ is an ideal of 
$K\langle X \rangle$ such that
$\phi(\Id(\A))\subseteq \Id(\A)$ for every endomorphism $\phi$ of $K\langle X \rangle$.
A T-ideal $I$ of $K\langle X \rangle$ generated by the polynomials 
$f_1, \ldots, f_{k} \in K\langle X \rangle$ 
is the minimal T-ideal of $K\langle X \rangle$ that contains 
$f_1, \ldots, f_k$. 
Denote
$I = \langle f_1, \ldots , f_k\rangle^{\mathrm{T}}$. We say that $f\in K\langle X \rangle$ is a consequence of $f_1, \ldots, f_k$ if $f\in I$.

For $x_1,x_2\in K\langle X\rangle$, we define the commutator $[x_1,x_2]=x_1x_2-x_2x_1$
and recursively
\[
[x_1,\ldots,x_n]
=
[[x_1,\ldots,x_{n-1}],x_n].
\]
Denote by $\mathcal P_n$ the space
of multilinear polynomials of degree $n$  in the free algebra $K\langle X \rangle$ and by $\mathcal{P}_n(\A)$ the quotient space
\[
\mathcal P_n(\A)=\frac{\mathcal P_n}{\mathcal{P}_n \cap \operatorname{Id}(\A)}.
\]
A polynomial $f\in K\langle X \rangle$ is called \textit{proper} if it is a
linear combination of products of commutators 
\[
f\left(x_1, \ldots, x_m\right)=\sum \alpha_{i, \ldots, j}\left[x_{i_1}, \ldots, x_{i_p}\right] \cdots\left[x_{j_1}, \ldots, x_{j_q}\right], \quad \alpha_{i, \ldots, j} \in K .
\]
We
denote by $\Gamma_n$ the space of \textit{multilinear proper polynomials of degree $n$} and by $\Gamma_n(\A)$ the quotient space
\[
\Gamma_n(\A)=\dfrac{\Gamma_n}{\Gamma_n\cap \Id(\A)}.
\]
Recall that the \textit{standard polynomial of degree
$n$} is
\[
\mathrm{st}_n\left(x_1, \ldots, x_n\right)=\sum_{\sigma \in \Sym_n} \operatorname{sign}(\sigma) x_{\sigma(1)} \cdots x_{\sigma(n)},
\]
where $\Sym_n$ is the symmetric group acting on the set $\{1, 2, \ldots , n\}$. 
By the Amitsur--Levitzki theorem~\cite{Amitsur&Levitzki1950},
$\mathrm{st}_4(x_1,x_2,x_3,x_4)$ is a polynomial identity of $M_2(K)$.
Moreover, $M_2(K)$ has no nonzero multilinear identity of degree less
than~$4$, and its space of multilinear identities of degree~$4$ is
spanned by $\mathrm{st}_4$.

\subsection*{The polynomial identities $\mathfrak{S}$ and $\mathfrak{T}$}

We recall a reduction criterion that will be used to prove that
$\mathfrak S$ and $\mathfrak T$ are polynomial identities of $M_2(K)$.
The following result is Lemma 5.1 from \cite{Asparouhov&Drensky&Koev&Tsiganchev2000}, presented in
our notation.

\begin{lemma}[\!\!\cite{Asparouhov&Drensky&Koev&Tsiganchev2000}]\label{lemma-generic-matrix-criterion}
Let $f(x_1,\ldots,x_m)\in\mathbb F_2\langle X\rangle$
be a proper multihomogeneous polynomial such that $\deg_{x_i}f=1$,  $i\geqslant 2$.
Then $f$ is a polynomial identity of the algebra
$R_m(\mathbb F_2)$ of generic $2\times2$ matrices if and only if $f(E_{11},z_2,\ldots,z_m)=0$
for all $z_i\in\{E_{11},E_{12},E_{21}\},$ $i=2,\ldots,m.$
\end{lemma}

Let $R_m(K)$ be the algebra generated by the generic $2 \times 2$ matrices $X_1,\ldots,X_m$. Since the polynomials considered below have
coefficients in the prime field \(\mathbb F_2\), an equality $f(X_1,\ldots,X_m)=0$
in \(R_m(\mathbb F_2)\) remains valid after scalar extension from
\(\mathbb F_2\) to \(K\).
Moreover, $R_{m}(K)$
is
isomorphic to the relatively free algebra of the variety generated by
$M_2(K)$~\cite{drensky2004polynomial}.
Consequently, for
\(f\in K\langle x_1,\ldots,x_m\rangle\), we have
\[
f\in\operatorname{Id}(M_2(K))
\quad\Longleftrightarrow\quad
f(X_1,\ldots,X_m)=0 \text{ in $R_m(K)$}.
\]
Therefore, Lemma~\ref{lemma-generic-matrix-criterion} provides a valid criterion for the polynomials
considered below.

\begin{lemma}\label{lemma-identities-M_2}
The polynomials~\eqref{ident-S} and~\eqref{ident-T} are identities for $M_2(K)$.
\end{lemma}

\begin{proof}
First, both $\mathfrak{S}$ and $\mathfrak{T}$ are proper multilinear polynomials of degree~$5$.
By Lemma~\ref{lemma-generic-matrix-criterion}, it is enough to verify the
evaluations
\[
x_1=E_{11},\qquad x_2,x_3,x_4,x_5\in\{E_{11},E_{12},E_{21}\}.
\]
Set
\[
a=E_{11},\qquad b=E_{12},\qquad c=E_{21},\qquad I=E_{11}+E_{22}.
\]
Using
$E_{ij}E_{kl}=\delta_{jk}E_{il}$, we obtain
\[
\begin{array}{c|ccc}
[\ ,\ ] & a & b & c \\ \hline
a & 0 & b & c\\
b & b & 0 & I\\
c & c & I & 0
\end{array}.
\]
First consider
\[
\mathfrak{T}(x_1,x_2,x_3,x_4,x_5)
=
[x_1,x_2,x_3][x_4,x_5]+
[x_1,x_2,x_4][x_5,x_3]+
[x_1,x_2,x_5][x_3,x_4].
\]
Putting \(x_1=a\), we have
\[
[a,x_2]\in\{0,b,c\}.
\]
If \([a,x_2]=0\), then \(\mathfrak{T}(a,x_2,x_3,x_4,x_5)=0\). Hence it remains to consider
\([a,x_2]=b\) and \([a,x_2]=c\).

The polynomial $\mathfrak{T}$ is symmetric in the variables \(x_3,x_4,x_5\), since a
permutation of these variables only permutes its three summands. Thus
it is enough to consider unordered triples \((x_3,x_4,x_5)\) with entries in
\(\{a,b,c\}\).

For \([a,x_2]=b\), the only nonzero cases are
\[
\begin{array}{c|c}
(x_3,x_4,x_5) & 
[b,x_3][x_4,x_5]+[b,x_4][x_5,x_3]+[b,x_5][x_3,x_4] \\ \hline
(a,a,c) & a+a\\
(a,b,c) & b+b\\
(a,c,c) & c+c\\
(b,c,c) & I+I
\end{array}
\]
and all these values are zero in characteristic \(2\). All remaining
unordered triples give zero term by term.

The case \([a,x_2]=c\) is analogous. The only nonzero cases are
\[
\begin{array}{c|c}
(x_3,x_4,x_5) & 
[c,x_3][x_4,x_5]+[c,x_4][x_5,x_3]+[c,x_5][x_3,x_4]  \\ \hline
(a,a,b) & E_{22}+E_{22}\\
(a,b,b) & b+b\\
(a,b,c) & c+c\\
(b,b,c) & I+I
\end{array}
\]
and again all values are zero. Hence $\mathfrak{T}(a,x_2,x_3,x_4,x_5)=0$
for all \(x_2,x_3,x_4,x_5\in\{a,b,c\}\).
Now consider
\[
\mathfrak{S}(x_1,x_2,x_3,x_4,x_5)=[[x_1,x_2][x_3,x_5],x_4]+[[x_1,x_3][x_4,x_5],x_2]+[[x_1,x_4][x_2,x_5],x_3].
\]
Putting \(x_1=a\), write the three summands as
\[
L_1=[[a,x_2][x_3,x_5],x_4],\qquad
L_2=[[a,x_3][x_4,x_5],x_2],\qquad
L_3=[[a,x_4][x_2,x_5],x_3].
\]
The polynomial $\mathfrak{S}$ is invariant under the cyclic permutation $(x_2,x_3,x_4)\mapsto (x_3,x_4,x_2).$
Therefore it is enough to consider representatives of the cyclic
orbits of triples $(x_2,x_3,x_4)\in\{a,b,c\}^3$.

The following table lists all representatives for which at least one
of \(L_1,L_2,L_3\) is nonzero. All omitted representatives give
\(L_1=L_2=L_3=0\).
\[
\begin{array}{c|c|c}
x_5 & (x_2,x_3,x_4) & (L_1,L_2,L_3) \\ \hline
a & (b,b,c) & (0,b,b)\\
a & (b,c,c) & (c,0,c)\\
b & (a,b,c) & (0,b,b)\\
b & (a,c,c) & (0,c,c)\\
b & (b,c,c) & (I,I,0)\\
c & (a,b,b) & (0,b,b)\\
c & (a,c,b) & (0,c,c)\\
c & (b,b,c) & (I,0,I)
\end{array}
\]
In every row the sum \(L_1+L_2+L_3\) is zero in characteristic \(2\).
Thus,  $\mathfrak{S}(a,x_2,x_3,x_4,x_5)=0$
for all admissible substitutions.
By Lemma~\ref{lemma-generic-matrix-criterion} and the scalar extension argument above, both $\mathfrak{S}$ and $\mathfrak{T}$ are
polynomial identities of $M_2(K)$.
\end{proof}

\begin{remark}
It is well known that the Hall identity
\[
[[x,y]^2,z]=0
\]
holds in the algebra $M_2(K)$ over any field. We observe that this identity is a consequence of~$\mathfrak S$.
Indeed, note that
\[
\mathfrak S(x,x,z,y,y)
=
[[x,x][z,y],y]
+
[[x,z][y,y],x]
+
[[x,y][x,y],z]=
[[x,y]^2,z],
\]
since $[x,x]=0$ and $[y,y]=0$. Therefore, $[[x,y]^2,z]\in \langle \mathfrak S\rangle^T.$
\end{remark}

\section{Proof of Theorem~\ref{theo-01}}\label{section-proof-thrm}

The computer algebra calculations in
the proofs of
Proposition~\ref{prop-02} and Theorem~\ref{theo-01} were carried out using Albert software~\cite{Jacobs1994Albert, Albert4M6}. More precisely, we computed the dimensions of the multilinear components of the relatively free algebras defined by the identities under consideration.

\begin{prop}\label{prop-02}
The polynomial $\mathfrak T$ is not a consequence of
$\mathrm{st}_4$ and $\mathfrak S$, and the polynomial
$\mathfrak S$ is not a consequence of
$\mathrm{st}_4$ and $\mathfrak T$.
\end{prop}
\begin{proof}
Using the software ``Albert"~\cite{Albert4M6, Jacobs1994Albert}, we obtain
\[
\dim \frac{\mathcal{P}_5}{\mathcal{P}_5\cap\langle \mathrm{st}_4,\mathfrak S\rangle^{\mathrm{T}}}=91,
\qquad
\dim \frac{\mathcal{P}_5}{\mathcal{P}_5\cap\langle\mathrm{st}_4,\mathfrak S,\mathfrak T\rangle^{\mathrm{T}}}=90.
\]
Hence $\mathfrak T$ is not a consequence of $\mathrm{st}_4$ and $\mathfrak S$.
Similarly, by ``Albert" we get
\[
\dim \frac{\mathcal{P}_5}{\mathcal{P}_5\cap\langle\mathrm{st}_4,\mathfrak T\rangle^{\mathrm{T}}}=91.
\]
Therefore, $\mathfrak S$ is not a consequence of $\mathrm{st}_4$ and $\mathfrak T$.
Consequently, $\mathfrak S$ and $\mathfrak T$ are independent modulo the multilinear
consequences of $\mathrm{st}_4$.
\end{proof}

Following the ideas from~\cite{Asparouhov&Drensky&Koev&Tsiganchev2000}, we developed a Python implementation to compute the dimensions of the spaces
\[
\Gamma_n\cap\Id(M_2(K)).
\]
See~\cite{Ferreira2026ProperIdentities2025} for the computer
program to assist in the proof of Lemma~\ref{lemma-dim-P6}.

\begin{lemma}\label{lemma-dim-P6}
Let $K$ be an infinite field of characteristic $2$. Then
\begin{itemize}
\item 
$\dim_{K}(\mathcal{P}_5\cap \Id(M_2(K)))
=
30;$
    \item $\operatorname{dim}_K\left(\mathcal{P}_6 \cap \Id(M_2(K))\right)=380;$
    \item 
    $\operatorname{dim}_K\left(\mathcal{P}_7 \cap \Id(M_2(K))\right)=3794$.
\end{itemize}
\end{lemma}

\begin{proof}
First, by~\cite[Lemma~$6$]{Drensky&Tsiganchev1998}, we have
\begin{equation}\label{eq-formula-dim-P}
\operatorname{dim}_K\left(\mathcal{P}_n \cap \Id(M_2(K))\right)=\sum_{k=0}^n\binom{n}{k} \operatorname{dim}_K\left(\Gamma_k \cap \Id(M_2(K))\right).
\end{equation}
Further, using a Python program~\cite{Ferreira2026ProperIdentities2025} we obtain that 
\[
\dim_{K}(\Gamma_5\cap \Id(M_2(K)))=25, 
\]
\[\dim_{K}(\Gamma_6\cap \Id(M_2(K)))=215, \]
\[\dim_{K}(\Gamma_7\cap \Id(M_2(K)))=1729.
\]
It is known that up to degree 4, the polynomial $\mathrm{st}_4(x_1,x_2,x_3,x_4)$ is the only proper multilinear identity of $M_2(K)$. Consequently, by~\eqref{eq-formula-dim-P} we obtain
\[
\dim_{K}(\mathcal{P}_5\cap\Id(M_2(K)))=\binom{5}{4}\cdot 1+\binom{5}{5}\cdot 25=5+25=30;
\]
\[
\dim_{K}(\mathcal{P}_6\cap\Id(M_2(K)))=\binom{6}{4}\cdot 1+\binom{6}{5}\cdot 25+\binom{6}{6}\cdot215=15+150+215=380;
\]
\begin{multline*}
\dim_K\left(\mathcal{P}_7 \cap \Id(M_2(K))\right)
=\binom{7}{4}\cdot1+\binom{7}{5}\cdot 25+\binom{7}{6}\cdot 215+\binom{7}{7}\cdot 1729=\\
=35+525+1505+1729=3794.
\end{multline*}
This completes the proof.
\end{proof}

\begin{remark}
The calculation in the proof of Lemma~\ref{lemma-dim-P6} was performed by a Python program that implements the following steps:
\begin{enumerate}
\item construct the Specht basis of $\Gamma_n$~\cite[Theorem~$4.3.9$]{Drensky2000};
\item form a generic element of $\Gamma_n$;
\item evaluate it on all substitutions of the basis matrices
$\{E_{11},E_{12},E_{21},E_{22}\}$ of $M_2(K)$;
\item compute the rank of the resulting homogeneous linear system.
\end{enumerate}
The implementation was validated by reproducing the dimensions in degrees $5$ and $6$ obtained in~\cite{Asparouhov&Drensky&Koev&Tsiganchev2000,Drensky&Tsiganchev1998}.
\end{remark}

\begin{proof_of}{of Theorem~\ref{theo-01}.}
Let $\mathcal I=\langle \mathrm{st}_4,\mathfrak S,\mathfrak T\rangle^T$
be the T-ideal generated by $\mathrm{st}_4,\mathfrak S$ and $\mathfrak T$, and put $W_n=\mathcal{P}_n\cap \mathcal{I}.$
By the Amitsur--Levitzki theorem~\cite{Amitsur&Levitzki1950}, $\mathrm{st}_4$ is a polynomial identity of $M_2(K)$. Moreover, by Lemma~\ref{lemma-identities-M_2}, the polynomials $\mathfrak S$ and $\mathfrak T$ are also identities of $M_2(K)$. Hence $\mathcal I\subseteq \operatorname{Id}(M_2(K)),$
and therefore
\[
W_n\subseteq \mathcal{P}_n\cap \operatorname{Id}(M_2(K))
\]
for every $n$.
It is well known that $M_2(K)$ has no nonzero multilinear identities of degree less than $4$, and that the multilinear identities of degree $4$ are generated by the standard polynomial~$\mathrm{st}_4$.
It remains to consider degrees $5,6$ and $7$. By computations with ``Albert"~\cite{Albert4M6, Jacobs1994Albert}, we obtain
\[
\dim_K\frac{\mathcal P_5}{W_5}=90,\qquad
\dim_K\frac{\mathcal P_6}{W_6}=340,\qquad
\dim_K\frac{\mathcal P_7}{W_7}=1246.
\]
Since $\dim_K \mathcal P_n=n!$, this gives
\[
\dim_K W_5=5!-90=30,
\]
\[
\dim_K W_6=6!-340=380,
\]
and
\[
\dim_K W_7=7!-1246=3794.
\]
On the other hand, by Lemma~\ref{lemma-dim-P6},
\[
\dim_K(\mathcal P_5\cap \operatorname{Id}(M_2(K)))=30,
\]
\[
\dim_K(\mathcal P_6\cap \operatorname{Id}(M_2(K)))=380,
\]
and
\[
\dim_K(\mathcal P_7\cap \operatorname{Id}(M_2(K)))=3794.
\]
Thus, for $n=5,6,7$, the inclusion
\[
W_n\subseteq \mathcal P_n\cap \operatorname{Id}(M_2(K))
\]
is an inclusion of vector spaces with the same dimension. Therefore,
\[
W_n=\mathcal P_n\cap \operatorname{Id}(M_2(K))
\]
for $n=5,6,7$. Thus, we conclude that all multilinear identities of degree at most $7$ follow from $\mathrm{st}_4,\mathfrak S$ and $\mathfrak T$.

It remains to prove that this generating set is minimal. By Proposition~\ref{prop-02} neither $\mathfrak S$ nor $\mathfrak T$ can be omitted.
Finally, $\mathrm{st}_4$ is not a consequence of $\mathfrak S$ and $\mathfrak T$. Indeed, $\mathfrak S$ and $\mathfrak T$ are multilinear proper polynomials of degree $5$. Since $\mathrm{st}_4$ spans 
$\mathcal P_4\cap \operatorname{Id}(M_2(K))$, it is not a consequence of $\mathfrak S$ and $\mathfrak T$. Thus the set $\{\mathrm{st}_4,\mathfrak S,\mathfrak T\}$
is minimal.
\end{proof_of}

\begin{prop}\label{prop-01}
Let $f$ be the multilinear identity of degree~$5$ for $M_2(K)$ found in~\cite{Asparouhov&Drensky&Koev&Tsiganchev2000}, namely
\begin{multline*}
f(x_1, x_2, x_3, x_4, x_5)= \left[x_5, x_4\right]\left[x_2, x_1, x_3\right]+\left[x_3, x_1, x_2\right]\left[x_5, x_4\right] +\\ 
 +\left[x_4, x_1, x_3\right]\left[x_5, x_2\right]+\left[x_3, x_1\right]\left[x_4, x_2, x_5\right] 
 +\left[x_4, x_3\right]\left[x_5, x_1, x_2\right]+\left[x_5, x_3, x_4\right]\left[x_2, x_1\right].
\end{multline*}
Then $f \in \langle \mathfrak S\rangle^{T}$.
Moreover, $\langle f,\mathfrak T\rangle^{T}
=
\langle \mathfrak S,\mathfrak T\rangle^{T}.$
\end{prop}
\begin{proof}
For simplicity, write
\[
\mathfrak S_{i_1i_2i_3i_4i_5}
=
\mathfrak S(x_{i_1},x_{i_2},x_{i_3},x_{i_4},x_{i_5}),
\]
and define analogously $\mathfrak T_{i_1i_2i_3i_4i_5}$ and
$f_{i_1i_2i_3i_4i_5}$.
First, using the derivation rule $[uv,w]=u[v,w]+[u,w]v,$
one obtains
\[
\begin{aligned}
\mathfrak S(x_1,x_2,x_3,x_4,x_5)
={}&[x_1,x_2][x_3,x_5,x_4]
+[x_1,x_2,x_4][x_3,x_5]\\
&+[x_1,x_3][x_4,x_5,x_2]
+[x_1,x_3,x_2][x_4,x_5]\\
&+[x_1,x_4][x_2,x_5,x_3]
+[x_1,x_4,x_3][x_2,x_5].
\end{aligned}
\]
A direct expansion gives
\begin{multline*}
    f=
\mathfrak S_{12435}
+\mathfrak S_{12453}
+\mathfrak S_{12534}
+\mathfrak S_{12543}
+\mathfrak S_{13452}
+\mathfrak S_{21345}
+\mathfrak S_{21453}
+\mathfrak S_{23451}
+\mathfrak S_{31245}\\
+\mathfrak S_{31524}
+\mathfrak S_{31542}
+\mathfrak S_{41235}
+\mathfrak S_{41352}.
\end{multline*}
Hence $f\in \langle \mathfrak S\rangle^T$.
Conversely, by direct expansion in the free associative algebra $K\langle X\rangle$, one has
\[
\mathfrak S_{12345}=F+G,
\]
where
\[
\begin{aligned}
F={}&
f_{12345}+f_{12354}+f_{12435}+f_{12453}
+f_{12534}+f_{12543}\\
&+f_{13245}+f_{13254}+f_{13425}+f_{13452}
+f_{13524}+f_{13542}\\
&+f_{14523}+f_{15234}+f_{15243}
+f_{23415}+f_{23514},
\end{aligned}
\]
and
\[
G=
\mathfrak T_{12345}
+\mathfrak T_{13245}
+\mathfrak T_{15234}
+\mathfrak T_{23145}.
\]
Therefore $\mathfrak S\in \langle f,\mathfrak T\rangle^T$.
The above decompositions were obtained by 
symbolic computation and verified by direct expansion in the free associative algebra $K\langle X \rangle$. The computations were verified by a Python implementation available at~\cite{Ferreira2026ProperIdentities2025}.
\end{proof}
\begin{cor}
Consider $f(x_1,x_2,x_3,x_4,x_5)\in K\langle X\rangle$ as in Proposition~\ref{prop-01}. The multilinear identities of degree at most $7$ for $M_2(K)$
are minimally generated by
\[
\mathrm{st}_4(x_1,x_2,x_3,x_4),\qquad f(x_1,x_2,x_3,x_4,x_5),\qquad \mathfrak T(x_1,x_2,x_3,x_4,x_5).
\]
\end{cor}
\begin{proof}
By Proposition~\ref{prop-01},
\[
f\in \langle\mathfrak S\rangle^{\mathrm T}
\qquad\text{and}\qquad
\mathfrak S\in\langle f,\mathfrak T\rangle^{\mathrm T}.
\]
Hence $\langle \mathrm{st}_4,\mathfrak S,\mathfrak T\rangle^{\mathrm T}=
\langle \mathrm{st}_4,f,\mathfrak T\rangle^{\mathrm T}$. Therefore, Theorem~\ref{theo-01} proves the generating assertion.
It remains to verify minimality. Clearly, $\operatorname{st}_4$
cannot be omitted, since the other two generators have degree~$5$
and are proper.
Suppose that $\mathfrak T\in\langle \operatorname{st}_4,f\rangle^T.$
Since $f\in\langle\mathfrak S\rangle^T$, this would imply $\mathfrak T\in
\langle \operatorname{st}_4,\mathfrak S\rangle^T,$
contrary to Proposition~\ref{prop-02}.
Similarly, if $f\in\langle \operatorname{st}_4,\mathfrak T\rangle^T,$
then, since $\mathfrak S\in\langle f,\mathfrak T\rangle^T,$
we would obtain $\mathfrak S\in
\langle \operatorname{st}_4,\mathfrak T\rangle^T,$
again contradicting Proposition~\ref{prop-02}. Hence the generating set is
minimal.
\end{proof}

\subsection*{Funding}%
This work was carried out as part of the collaboration project between
the Universidade Estadual de Campinas (IMECC--UNICAMP) and the
Universidade Federal do Rio Grande do Norte (DMAT--UFRN).
The author was supported by the Coordenação de Aperfeiçoamento de
Pessoal de Nível Superior -- Brasil (CAPES) -- Finance Code 001.

\subsection*{Acknowledgments}%
The author is grateful to the participants of the workshop
``Matrix Identities,'' organized by the research net
``Combinatorics in Algebraic Structures: UFRN--UNICAMP--UFCG,''
for useful discussions and suggestions that contributed to the
improvement of this note.

\section*{Disclosure statement}

The author declares that there are no relevant financial or non-financial competing interests to report.

{
\small
\bibliography{reference}

\bibliographystyle{abbrvurl}
}

\end{document}